\documentclass[12pt,a4paper,reqno]{amsart}
\usepackage{amsthm,amsfonts,amsmath,amscd,amssymb}
\usepackage{latexsym}
\usepackage{euscript}
\usepackage{enumitem}
\usepackage{graphicx}
\usepackage{tikz}
\usepackage{caption}
\usepackage{subcaption}
\usepackage{cite}
\usepackage[utf8]{inputenc}
\usepackage[english]{babel}
\usepackage{xcolor}

\usepackage{amssymb}
\usepackage[utf8]{inputenc}
\usepackage[T2A]{fontenc}
\usepackage[all]{xy}
\usepackage[margin=2.8cm]{geometry}

\begin{document}



\title[Automorphisms of cellular divisions of $2$-sphere]{Automorphisms of cellular divisions of $2$-sphere induced by functions with isolated critical points}



\author{Anna Kravchenko}
\address{Taras Shevchenko National University of Kyiv, Ukraine}
\email{annakravchenko1606@gmail.com}

\author{Sergiy Maksymenko}
\address{Institute of Mathematics, National Academy of Sciences of Ukraine, Kyiv, Ukraine}
\email{maks@imath.kiev.ua}

\newtheorem{theorem}[subsection]{Theorem}
\newtheorem{lemma}[subsection]{Lemma}
\newtheorem{sublemma}[subsubsection]{Lemma}
\newtheorem{proposition}[subsection]{Proposition}
\newtheorem{corollary}[subsection]{Corollary}
\newtheorem{subcorollary}[subsubsection]{Corollary}

\theoremstyle{definition}
\newtheorem{definition}[subsection]{Definition}
\newtheorem{subdefinition}[subsubsection]{Definition}

\theoremstyle{remark}
\newtheorem{remark}[subsection]{Remark}
\newtheorem{example}[subsection]{Example}
\newtheorem{notation}[subsection]{Notation}

\makeatletter
\@addtoreset{subsection}{section}
\@addtoreset{equation}{section}
\@addtoreset{figure}{section}
\@addtoreset{table}{section}
\makeatother
\renewcommand{\theequation}{\thesection.\arabic{equation}}
\renewcommand{\thefigure}{\thesection.\arabic{figure}}
\renewcommand{\thetable}{\thesection.\arabic{table}}

\makeatletter
\newcommand\testshape{family=\f@family; series=\f@series; shape=\f@shape.}
\def\myemphInternal#1{\if n\f@shape%
	\begingroup\itshape #1\endgroup\/%
	\else\begingroup\sf\itshape #1\endgroup%
	\fi}
\def\myemph{\futurelet\testchar\MaybeOptArgmyemph}
\def\MaybeOptArgmyemph{\ifx[\testchar \let\next\OptArgmyemph
	\else \let\next\NoOptArgmyemph \fi \next}
\def\OptArgmyemph[#1]#2{\index{#1}\myemphInternal{#2}}
\def\NoOptArgmyemph#1{\myemphInternal{#1}}
\makeatother

\newcommand\warn[1]{#1}


\newcommand\id{\mathrm{id}}          
\newcommand\IM{\mathrm{Im}}          
\newcommand\RE{\mathrm{Re}}          
\newcommand\Int{\mathrm{Int}}        
\newcommand\Per{\mathrm{Per}}        
\newcommand\Fix[1]{\mathrm{Fix}(#1)} 
\newcommand\supp{\mathrm{supp\,}}    
\newcommand\grad{\triangledown}      
\newcommand\rank{\mathrm{rank}}      

\newcommand\wrm[1]{\mathop{\wr}\limits_{#1}} 
\newcommand\scprod[2]{\langle#1,#2\rangle} 


\newcommand\eps{\varepsilon}


\newcommand\bR{\mathbb{R}}
\newcommand\bD{\mathbb{D}}
\newcommand\bS{\mathbb{S}}
\newcommand\bA{\mathbb{A}}
\newcommand\bZ{\mathbb{Z}}
\newcommand\bN{\mathbb{N}}
\newcommand\bQ{\mathbb{Q}}
\newcommand\bC{\mathbb{C}}

\newcommand\Aut{\mathrm{Aut}}       
\newcommand\Aff{\mathrm{Aff}}       
\newcommand\Diff{\mathcal{D}}       
\newcommand\Homeo{\mathcal{H}}      
\newcommand\End{\mathcal{E}}        
\newcommand\Map{\mathrm{Map}}       
\newcommand\Orb{\mathcal{O}}        
\newcommand\Stab{\mathcal{S}}       
\newcommand\Maps[2]{#2^{#1}}        

\newcommand\DiffId{\Diff_{\id}}     
\newcommand\HomeoId{\Homeo_{\id}}   
\newcommand\StabId{\Stab_{\id}}     


\newcommand\Cinfty{\mathcal{C}^{\infty}}
\newcommand\Cont[2]{\mathcal{C}\left(#1,#2\right)}           
\newcommand\Ci[2]{\mathcal{C}^{\infty}(#1,#2)}               
\newcommand\Cid[2]{\mathcal{C}_{\partial}^{\infty}(#1,#2)}   
\newcommand\Morse[2]{\mathcal{M}(#1,#2)}                     
\newcommand\MorseSmp[2]{\mathcal{M}^{smp}(#1,#2)}            
\newcommand\MorseGen[2]{\mathcal{M}^{gen}(#1,#2)}            


\newcommand\Stabilizer[1]{\Stab(#1)}
\newcommand\StabilizerId[1]{\StabId(#1)}
\newcommand\StabilizerIsotId[1]{\Stab'(#1)}
\newcommand\Orbit[1]{\Orb(#1)}
\newcommand\OrbitComp[2]{\Orb_{#1}(#2)}


\newcommand\GL{\mathrm{GL}}
\newcommand\SL{\mathrm{SL}}
\newcommand\SO{\mathrm{SO}}
\newcommand\UnitGroup{\{1\}}


\newcommand\Aman{A}
\newcommand\Bman{B}
\newcommand\Cman{C}
\newcommand\Dman{D}
\newcommand\Eman{E}
\newcommand\Fman{F}
\newcommand\Gman{G}
\newcommand\Hman{H}
\newcommand\Iman{I}
\newcommand\Jman{J}
\newcommand\Kman{K}
\newcommand\Lman{L}
\newcommand\Mman{M}
\newcommand\Nman{N}
\newcommand\Oman{O}
\newcommand\Pman{P}
\newcommand\Qman{Q}
\newcommand\Rman{R}
\newcommand\Sman{S}
\newcommand\Tman{T}
\newcommand\Uman{U}
\newcommand\Vman{V}
\newcommand\Wman{W}
\newcommand\Xman{X}
\newcommand\Yman{Y}
\newcommand\Zman{Z}

\newcommand\cov[1]{\tilde{#1}} %
\newcommand\tAman{\cov{\Aman}}
\newcommand\tBman{\cov{\Bman}}
\newcommand\tCman{\cov{\Cman}}
\newcommand\tDman{\cov{\Dman}}
\newcommand\tEman{\cov{\Eman}}
\newcommand\tFman{\cov{\Fman}}
\newcommand\tGman{\cov{\Gman}}
\newcommand\tHman{\cov{\Hman}}
\newcommand\tIman{\cov{\Iman}}
\newcommand\tJman{\cov{\Jman}}
\newcommand\tKman{\cov{\Kman}}
\newcommand\tLman{\cov{\Lman}}
\newcommand\tMman{\cov{\Mman}}
\newcommand\tNman{\cov{\Nman}}
\newcommand\tOman{\cov{\Oman}}
\newcommand\tPman{\cov{\Pman}}
\newcommand\tQman{\cov{\Qman}}
\newcommand\tRman{\cov{\Rman}}
\newcommand\tSman{\cov{\Sman}}
\newcommand\tTman{\cov{\Tman}}
\newcommand\tUman{\cov{\Uman}}
\newcommand\tVman{\cov{\Vman}}
\newcommand\tWman{\cov{\Wman}}
\newcommand\tXman{\cov{\Xman}}
\newcommand\tYman{\cov{\Yman}}
\newcommand\tZman{\cov{\Zman}}

\newcommand\UU{\mathcal{U}}
\newcommand\VV{\mathcal{V}}
\newcommand\WW{\mathcal{W}}


\newcommand\DiffM{\Diff(\Mman)}
\newcommand\DiffIdM{\DiffId(\Mman)}
\newcommand\DiffMX{\Diff(\Mman,\Xman)}

\newcommand\func{f}
\newcommand\gfunc{g}
\newcommand\dif{h}
\newcommand\gdif{g}
\newcommand\qdif{q}

\newcommand\fSing{\Sigma_{\func}}

\newcommand\afunc{\alpha}
\newcommand\bfunc{\beta}
\newcommand\cfunc{\gamma}
\newcommand\dfunc{\delta}


\newcommand\Grp{\mathbf{G}}
\newcommand\Grpf[1]{\Grp(#1)}
\newcommand\Gf{\Grpf{\func}}
\newcommand\Gg{\Grpf{\gfunc}}
\newcommand\Ga{\Grpf{\alpha}}
\newcommand\Gb{\Grpf{\beta}}

\newcommand\sA{X}
\newcommand\sB{Y}
\newcommand\sL{L^{*}}
\newcommand\sG{G_{f}}
\newcommand\sF{F(D^{2})}
\newcommand\tG{\varGamma}
\newcommand\Sp{S^{2}}
\newcommand\fG{G_f}
\newcommand\vG{G_{v}}
\newcommand\lG{G_{v}^{loc}}

\newcommand\cw{\mathrm{cw}}
\newcommand\Is{\mathrm{Isom}^{\cw}}
\newcommand\XX{X^{1}}
\newcommand\Xk{X^{k}}
\newcommand\Xkk{X^{k+1}}
\newcommand\kk{k}
\newcommand\Cir{S^{1}}

\newcommand\MorseSmpD{\MorseSmp{D^2}{\bR}}
\newcommand\MorseSmpM{\MorseSmp{\Mman}{\bR}}
\newcommand\PClassAll{\mathcal{P}}
\newcommand\PClassTwo{\PClassAll_2}

\newcommand\ClassGf{\mathbf{G}}
\newcommand\ClassGfM{\ClassGf(\Mman)}
\newcommand\ClassGfSmpM{\ClassGf^{smp}(\Mman)}
\newcommand\ClassGfGenM{\ClassGf^{gen}(\Mman)}

\newcommand\ClassGfSmpX[1]{\ClassGf^{smp}(#1)}
\newcommand\ClassGfSmpDisk{\ClassGfSmpX{D^2}}

\newcommand\KRGraph[1]{\Delta_{#1}}
\newcommand\KRGraphf{\KRGraph{\func}}
\newcommand\word{W}
\newcommand\aword{A}
\newcommand\bword{B}
\newcommand\len{\mathrm{len}}
\newcommand\wL{\len(\word)}

\newcommand\aGrp{G}
\newcommand\medge[1]{e(#1)}

\newcommand\RR{\mathcal{R}}
\newcommand\GG{\mathcal{G}}

\newcommand\msep{ \ \ }

\newcommand\bSet{B}
\newcommand\bGroup{H}
\newcommand\actHom{\delta}
\newcommand\aStab[1]{S_{#1}}
\newcommand\hh{h}

\newcommand\GRP{\mathcal{G}}
\newcommand\GStabilizer[1]{\Stab^{\,\GRP}(#1)}
\newcommand\GStabilizerId[1]{\GRP\StabId(#1)}
\newcommand\GStabilizerIsotId[1]{\GRP\Stab'(#1)}
\newcommand\FClass{\mathcal{Z}}

\newcommand\Partit{\Xi}
\newcommand\HP{\Homeo(\Partit)}
\newcommand\HZP{\Homeo_{0}(\Partit)}

\newcommand\gel{\gamma}
\newcommand\sgel{\widehat{\gel}}
\newcommand\emb{\xi}

\newcommand\metr{d}

\keywords{surface, Morse function, diffeomorphisms}


\maketitle

\begin{abstract}
Let $f:S^2\to \mathbb{R}$ be a Morse function on the $2$-sphere and $K$ be a connected component of some level set of $f$ containing at least one saddle critical point.
Then $K$ is a $1$-dimensional CW-complex cellularly embedded into $S^2$, so the complement $S^2\setminus K$ is a union of open $2$-disks $D_1,\ldots, D_k$.
Let $\mathcal{S}_{K}(f)$ be the group of isotopic to the identity diffeomorphisms of $S^2$ leaving invariant $K$ and also each level set $f^{-1}(c)$, $c\in\mathbb{R}$.
Then each $h\in \mathcal{S}_{K}(f)$ induces a certain permutation $\sigma_{h}$ of those disks.
Denote by $G = \{ \sigma_h \mid h \in \mathcal{S}_{K}(f)\}$ be the group of all such permutations.
We prove that $G$ is isomorphic to a finite subgroup of $SO(3)$.

\end{abstract}

\section{Introduction}
Study of groups of automorphisms of discrete structures has a long history.
One of the first general results obtained by A.~Cayley (1854) claims that every finite group $G$ of order $n$ is a subgroup of the permutation group of a set consisting of $n$ elements, see also E.~Nummela~\cite{Nummela:AMM:1980} for extension of this fact to topological groups.
C.~Jordan~\cite{Jordan:JRAM:1869} (1869) described the structure of groups of automorphisms of finite trees and R.~Frucht~\cite{Frucht:CM:1939} (1939) shown that every finite group can also be realized as a group of symmetries of certain finite graph.

Given a closed compact surface $\Mman$ endowed with a cellular decomposition $\Partit$, (e.g. with a triangulation) one can consider the group of ``combinatorial'' automorphisms of $\Mman$.
More precisely, say that a homeomorphism $\dif:\Mman\to\Mman$ is \myemph{cellular} or a \myemph{$\Partit$-homeomorphism}, if it maps $i$-cells to $i$-cells, and $\dif$ is \myemph{$\Partit$-trivial} if it preserves every cell with its orientation.
Then the group of \myemph{combinatorial automorphisms} of $\Partit$ is the group of $\Partit$-homeomorphisms modulo $\Partit$-trivial ones.
Denote this group by $\Aut(\Partit)$.
It was proved by R.~Cori and A.~Machi~\cite{CoriMachi:TCS:1982} and J.~{\v{S}}ir{\'a}{\v{n}} and M.~{\v{S}}koviera~\cite{SiranSkoviera:AJC:1993} that every finite group is isomorphic with $\Aut(\Partit)$ for some cellular decomposition of some surface which can be taken equally either orientable or non-orientable.

Notice that the $1$-skeleton $\Mman^{1}$ of $\Partit$ can be regarded as a graph.
Suppose each vertex $\Mman^{1}$ has even degree.
Then in many cases one can construct a smooth ($\Cinfty$) function $\func:\Mman\to\bR$ such that $\Mman^{1}$ is a critical level contaning \myemph{all} saddles (i.e.\! critical points being not local extremes), and the group $\Aut(\Partit)$ can be regarded as the group of ``\myemph{combinatorial  symmetries}'' of $\func$.

Such a point of view was motivated by works of A.~Fomenko on classification of Hamiltonian systems, see~\cite{Fomenko:RANDAN:1986, Fomenko:UMN:1989}.
The group $\Aut(\Partit)$ is called \myemph{the group of symmetries of an ``atom'' of $\func$}.
Such groups for the case when $\func$ is a Morse function were studied by A.~Fomenko and A.~Bolsinov~\cite{BolsinovFomenko:1998},
A.~Oshemkov and Yu.~Brailov~\cite{Brailov:1998}, Yu.~Brailov and E.~Kudryavtseva~\cite{BrailovKudryavtseva:VMNU:1999}, A.~A. Kadubovsky and A.~V. Klimchuk~\cite{KadubovskyKlimchuk:MFAT:2004}, and A.~Fomenko, E.~Kudryavtseva and I.~Nikonov~\cite{KudryavtsevaNikonovFomenko:MS:2008}.

In~\cite{Maksymenko:AGAG:2006} the author gave sufficient conditions for a $\Partit$-homeomorphism to be $\Partit$-trivial, and in~\cite{Maksymenko:MFAT:2010} estimated the number of invariant cells of a $\Partit$-homeomorphism.

It was proved by A.~Fomenko and E.~Kudryavtseva~\cite{KudryavtsevaFomenko:DANRAN:2012, KudryavtsevaFomenko:VMU:2013} that every finite group is the group of combinatorial symmetries of some Morse function $\func$ on some compact orientable surface having critical level containing all saddles.
However, the number of critical points of $\func$ as well as the genus of $\Mman$ can be arbitrary large.

In general, if $\func:\Mman\to\bR$ is an arbitrary smooth function with isolated critical points, then a certain part of its ``combinatorial symmetries'' is reflected by a so-called \myemph{Kronrod-Reeb} graph $\KRGraphf$, see e.g.~\cite{Kronrod:UMN:1950, BolsinovFomenko:1997, Kulinich:MFAT:1998, Kudryavtseva:MatSb:1999, Sharko:UMZ:2003, Polulyakh:UMJ:2016, Michalak:TMNA:2018, BatistaCostaMezaSarmiento:JS:2018} and \S\ref{sect:KRGraph}.
Such a graph is obtained by shrinking each connected component of each level set $\func^{-1}(c)$, $c\in\bR$, of $\func$ into a point.

Let $\DiffM$ the group of diffeomorphisms of $\Mman$ and
\[ 
\Stabilizer{\func} =\{\dif\in\DiffM \mid \func(\dif(x))=\func(x) \ \text{for all} \ x\in\Mman\}
\]
be the group of diffeomorphisms $\dif$ of $\Mman$ which ``preserve'' $\func$ in the sense that $\dif$ leaves invariant each level set $\func^{-1}(c)$, $c\in\bR$, of $\func$.
Hence it yields a certain permutation of connected components of $\func^{-1}(c)$ being points of $\KRGraphf$, and thus induces a certain map $\rho(\dif):\KRGraphf\to\KRGraphf$.
It can be shown that $\rho(\dif)$ is a homeomorphism of $\KRGraphf$, and the correspondence $\rho:\dif\mapsto\rho(\dif)$ is a \myemph{homomorphism} of groups
\[\rho :\Stabilizer{\func} \to \Homeo(\KRGraphf),\]
where $\Homeo(\KRGraphf)$ is the group of homeomorphisms of $\KRGraphf$.
One can also verify that the image of $\rho(\Stabilizer{\func})$ is a \myemph{finite} group.

Let also $\DiffIdM$ be the identity path component of $\DiffM$, and
\[ \StabilizerIsotId{\func} = \Stabilizer{\func} \cap \DiffIdM   \]
be the group of $\func$-preserving diffeomorphisms which are isotopic to the identity via an isotopy consisting of not necessarily $\func$-preserving diffeomorphsms.
We will be interested in the group
\[
    \fG =\rho(\StabilizerIsotId{\func})
\]
of automorphisms of $\KRGraphf$ induced by elements from $\StabilizerIsotId{\func}$, see Remark~\ref{rem:fG} for the structure and applications of $\fG$. 

Suppose that the set $\Fix{\fG}$ of common fixed points of all elements of $\fG$ in $\KRGraphf$ is non-empty. 
Let also $v \in \Fix{\fG}$ be a vertex of $\KRGraphf$ fixed under $\fG$ and $Star(v)$ be a \myemph{star} of $v$, i.e. a small $\fG$-invariant neighborhood of $v$.
Then each $\gel\in\fG$ induces a homeomorphism of $Star(v)$, and we can also define the group 
\[\lG=\{\gel|_{Star(v)} \mid \gel \in \fG\}\] of restrictions of elements of $\fG$ to $Star(v)$.
We will call $\lG$ the \myemph{local stabilizer} of $v$.

\begin{remark}\label{rem:Gvloc}
We will give now an equivalent description of the group $\lG$.
Let $K$ be the critical component of a level-set of $\func$ corresponding to the vertex $v\in\KRGraphf$.
Since $v\in\Fix{\fG}$, we obtain that $\dif(K)=K$ for all $\dif\in\StabilizerIsotId{\func}$.
Let $c = \func(K)$ be the value of $\func$ on $K$, and $\varepsilon>0$ be a small number such that the segment $[c-\varepsilon, c+\varepsilon]$ contains no other critical values of $\func$ except for $c$.
Let also $N_{K}$ be the connected component of $\func^{-1}[c-\varepsilon, c+\varepsilon]$ containing $K$.
Notice that the quotient map $p$ induces a bijection between connected components $\partial N_{K}$ and edges of $Star(v)$.
Moreover, $\dif(N_{K})=N_{K}$ for all $\dif\in\StabilizerIsotId{\func}$, and hence $\dif$ induces a permutation $\sigma_{\dif}$ of connected components of $\partial N_{K}$.
Then $\lG$ is the same as the group of permutations of connected components of $\partial N_{K}$ induced by $\dif$.
\end{remark}

In \cite{Maksymenko:DefFuncI:2014, MaksymenkoFeshchenko:UMZ:ENG:2014, MaksymenkoFeshchenko:MS:2015, MaksymenkoFeshchenko:MFAT:2015, MaksymenkoKravchenko:GMF:2018}, the groups $\lG$ were calculated for all Morse functions on all orientable surfaces distinct from $\Sp$.
In the present paper, we give a complete description of the structure of the group $\lG$ to the case when $\Mman=\Sp$.
For the convenience of the reader we present a general statement about the structure of the group $\lG$ for all orientable surfaces.

\begin{theorem}\label{th:main}
Let $\func\in\Ci{\Mman}{\bR}$ be a Morse function and $v\in\Fix{\fG}$ be some vertex.
\begin{enumerate}[label={\rm(\arabic*)}, leftmargin=*, itemsep=1ex]
\item\label{enum:mm1}
If $\Mman\neq\Sp,T^{2}$, then $\lG\approx\bZ_{n}$, for some $n\geq1$, {\rm\cite{Maksymenko:DefFuncI:2014}}.

\item\label{enum:mm2}
If $\Mman=T^{2}$, then $\lG\approx\bZ_{m}\times\bZ_{mn}$, for some $m,n\geq1$, {\rm\cite{MaksymenkoFeshchenko:UMZ:ENG:2014, MaksymenkoFeshchenko:MS:2015, MaksymenkoFeshchenko:MFAT:2015}}.

\item\label{enum:mm3}
Let $\Mman=\Sp$. Then the following statements hold.
\begin{enumerate}[label={\rm(\alph*) }, leftmargin=*]
\item\label{enum:m1}
For each vertex $v\in\Fix{\fG}$, the group $\lG$  is isomorphic to a finite subgroup of $SO(3)$, that is, to one of the following groups, see~{\rm\cite[pp.~21-23]{Klein:1888}:}
\begin{equation}\label{eq:mh}
	\bZ_{n},  \    \ \bD_{n},  \    \ \bA_{4}, \    \ \bS_{4},  \   \ \bA_{5},  \   \ (n\geq1).
\end{equation}
\item\label{enum:m3}
If $\Fix{\fG}$ has at least one edge, then for any vertex $v\in\Fix{\fG}$, the group $\lG$ is cyclic.

\item\label{enum:m4}
If $\Fix{\fG}$ consists of a unique vertex $v$ and $\lG$ is non-trivial and cyclic, then $\lG \cong \bZ_2$.
\end{enumerate}
\end{enumerate}
\end{theorem}

We need to prove only the item~\ref{enum:mm3} of this theorem.
In fact we will establish a more general result including~\ref{enum:mm3} of Theorem~\ref{th:main} as a partial case, see Theorem~\ref{th:cw3} and \S\ref{sect:proof:th:main}.

\begin{remark}\label{rem:fG}\rm
Notice that $\Stabilizer{\func}$ can be regarded as the \myemph{stabilizer} of $\func$ with respect to the natural right action
$\phi:\Ci{\Mman}{\bR}\times\DiffM\to\Ci{\Mman}{\bR}$ of $\DiffM$ on the space $\Ci{\Mman}{\bR}$ of smooth functions on $\Mman$ defined by $\phi(\func,\dif) = \func \circ \dif$.
Then the group $\fG$ plays an key role in determining the homotopy type of the orbit $\Orbit{\func} = \{ \func\circ\dif \mid \dif\in\DiffM \}$ of $\func$ with respect to the above action, see E.~Kudryavtseva~\cite{Kudryavtseva:SpecMF:VMU:2012, Kudryavtseva:MathNotes:2012, Kudryavtseva:MatSb:2013}, S.~Maksymenko~\cite{Maksymenko:AGAG:2006, Maksymenko:DefFuncI:2014}.

A function $f \in \Ci{\Mman}{\bR}$ is \myemph{Morse} if the following conditions are fulfilled:
\begin{enumerate}[label=$\bullet$]
    \item $\func$ takes constant values on the connected components of the boundary $\partial\Mman$;
    \item each critical point $\func$ is nondegenerate and is contained in $\Int{\Mman}$.
\end{enumerate}
We will denote by $\Morse{M}{\bR}$ the set of all Morse functions on $\Mman$.

Let $\mathcal{P}$ be the minimal set of isomorphism classes of groups satisfying the following conditions:
\begin{enumerate}[label=$\bullet$]
    \item a unit group $\{1\}\in \PClassAll$;
    \item if $A, B\in \PClassAll$ and $n\in\bN$, then  $A\times B, \ A\wr \bZ_{n} \, \in \, \PClassAll$,
\end{enumerate}
where $A\wr \bZ_{n}$ is the \myemph{wreath product} of groups $A$ and $\bZ_{n}$ which can be defined as the \myemph{direct product of sets} $\underbrace{A\times\cdots\times A}_{n} \times \bZ_{n}$ with the following operation:
\[
    (a_0,a_1,\ldots,a_{n-1}, k) \,
    (b_0,b_1,\ldots,b_{n-1}, l)  =
    (a_0 b_k,\, a_1 b_{k+1},\, \ldots, \,a_{n-1} b_{k-1}, \, k+l),
\]
where all indices are taken modulo $n$.

In~\cite{MaksymenkoKravchenko:GMF:2018} the authors described the structure of the set
\begin{equation}
G({M}, {\mathbb{R}})=\{{\fG} \mid {f} \in \Morse{M}{\bR}\}
\end{equation}
of groups $\fG$ for all Morse functions on all orientable surfaces distinct from 2-sphere and 2-torus.
It was proved that $G({M}, {\bR})=\mathcal{P}$.
The structure of the groups $G({T^{2}}, {\mathbb{R}})$ and $G({\Sp}, {\mathbb{R}})$ will be studied in forthcoming papers.
\end{remark}

\section{Main result}\label{sect:main_result}
\subsection{Isolated critical points of smooth functions on the plane.}
Let $\func:\bR^2 \to \bR$ be a smooth function such that $0\in\bR^2$ is an isolated critical point of $\func$.
Then there exists an open neighborhood $U$ of $0$ in $\bR^2$ and a topological embedding (homeomorphism onto its image) $\dif: U \to \bR^2$ such that $\dif(0)=0$, and the composition $\func\circ \dif :U \to \bR$ is given by one of the following formulas:
\[
\func\circ \dif(x,y) =
\begin{cases}
    \pm(x^2+y^2), & \text{if $0\in\bR^2$ is a local extreme of $\func$, \cite{Dancer:2:JRAM:1987}}, \\
    Im( (x+iy)^k ), & \text{for some $k\geq1$ otherwise, \cite{Prishlyak:TA:2002}}.
\end{cases}
\]
The structures of level sets of $\func$ near an isolated critical points are shown in Figures~\ref{fig:loc_extreme}, \ref{fig:saddle_index_1}, \ref{fig:saddle_index_3}.

\begin{figure}[ht]
\centering\includegraphics[height=2.5cm]{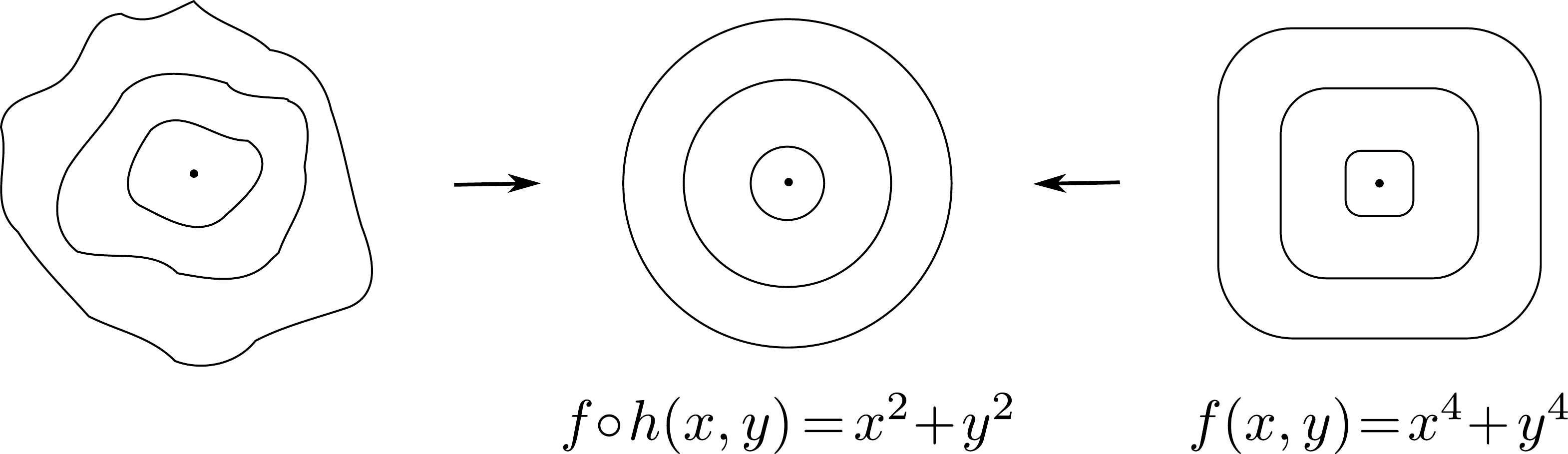}
\caption{Local extreme}
\label{fig:loc_extreme}
\end{figure}
\begin{figure}[ht]
\centering\includegraphics[height=2.5cm]{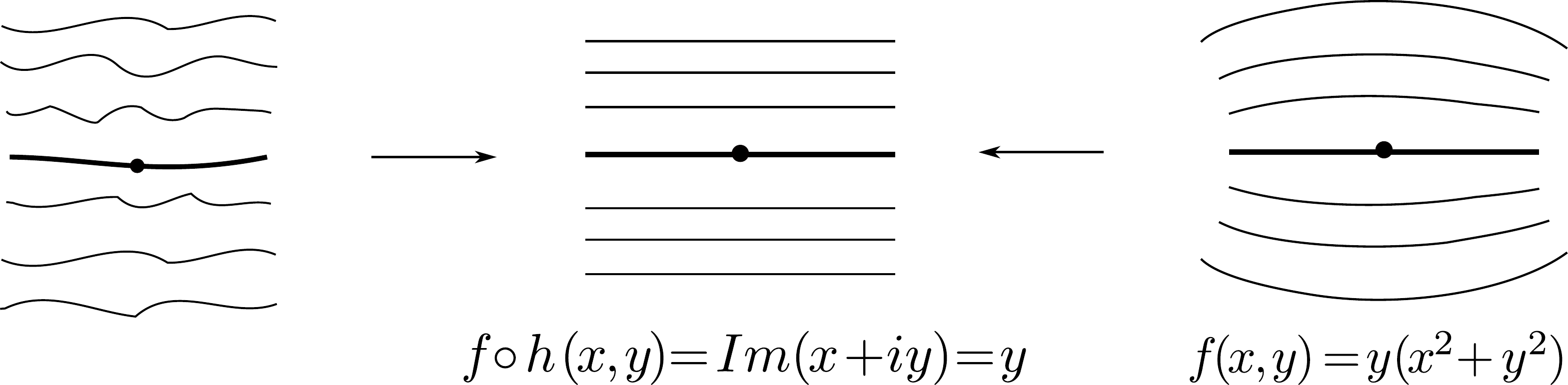}
\caption{Saddle point of order $1$}
\label{fig:saddle_index_1}
\end{figure}
\begin{figure}[ht]
\centering\includegraphics[height=3.2cm]{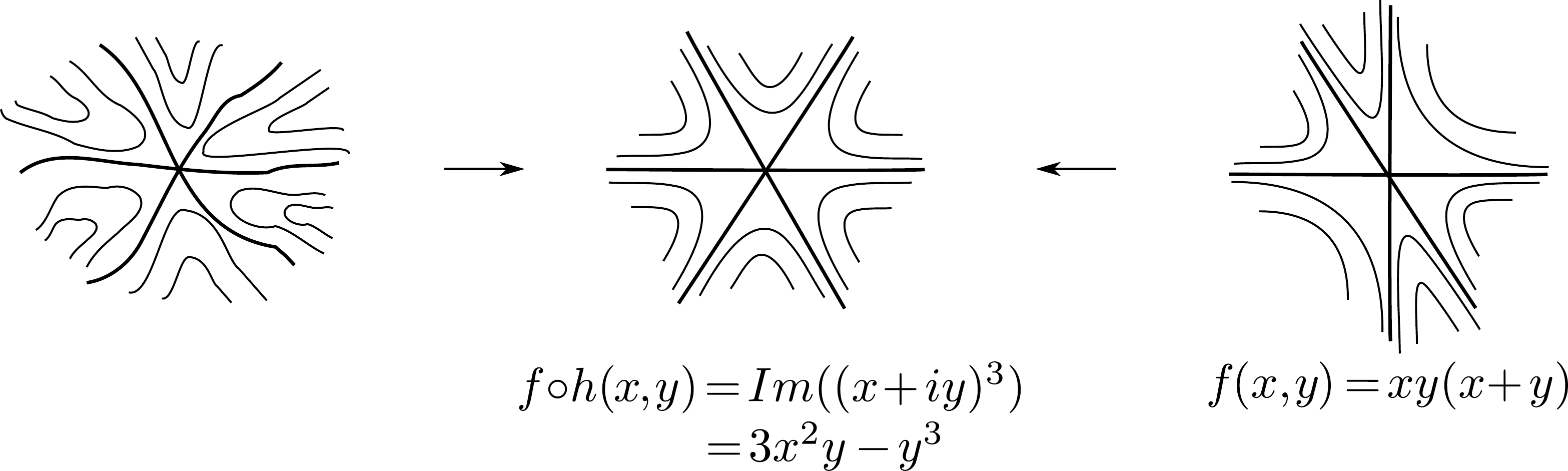}
\caption{Saddle point of order $3$}
\label{fig:saddle_index_3}
\end{figure}

In particular, if $0$ is a local extreme, then the level sets of $\func$ are concentric simple closed curves wrapping around $0$.

Otherwise, $0$ is called a \myemph{saddle}, and the critical level set of $\func$ near $0$ consists of $2k$ arcs
\[ \alpha_0, \alpha_{1}, \ldots, \alpha_{2k-1} \]
starting from $0$.
They split a neighborhod of $0$ into $2k$-sectors $\widehat{\alpha_{i} \alpha_{i+1}}$ so that the values of $\func$ in the consecutive sectors $\widehat{\alpha_{i-1} \alpha_{i}}$ and $\widehat{\alpha_{i} \alpha_{i+1}}$ are of opposite signs, see Figure~\ref{fig:saddle_sectors}.
In particular, the number $k$ does not depend on a paticular choice of $\dif$, and will be called the \myemph{order} of the critical point.

\begin{figure}[ht]
\centering\includegraphics[height=3cm]{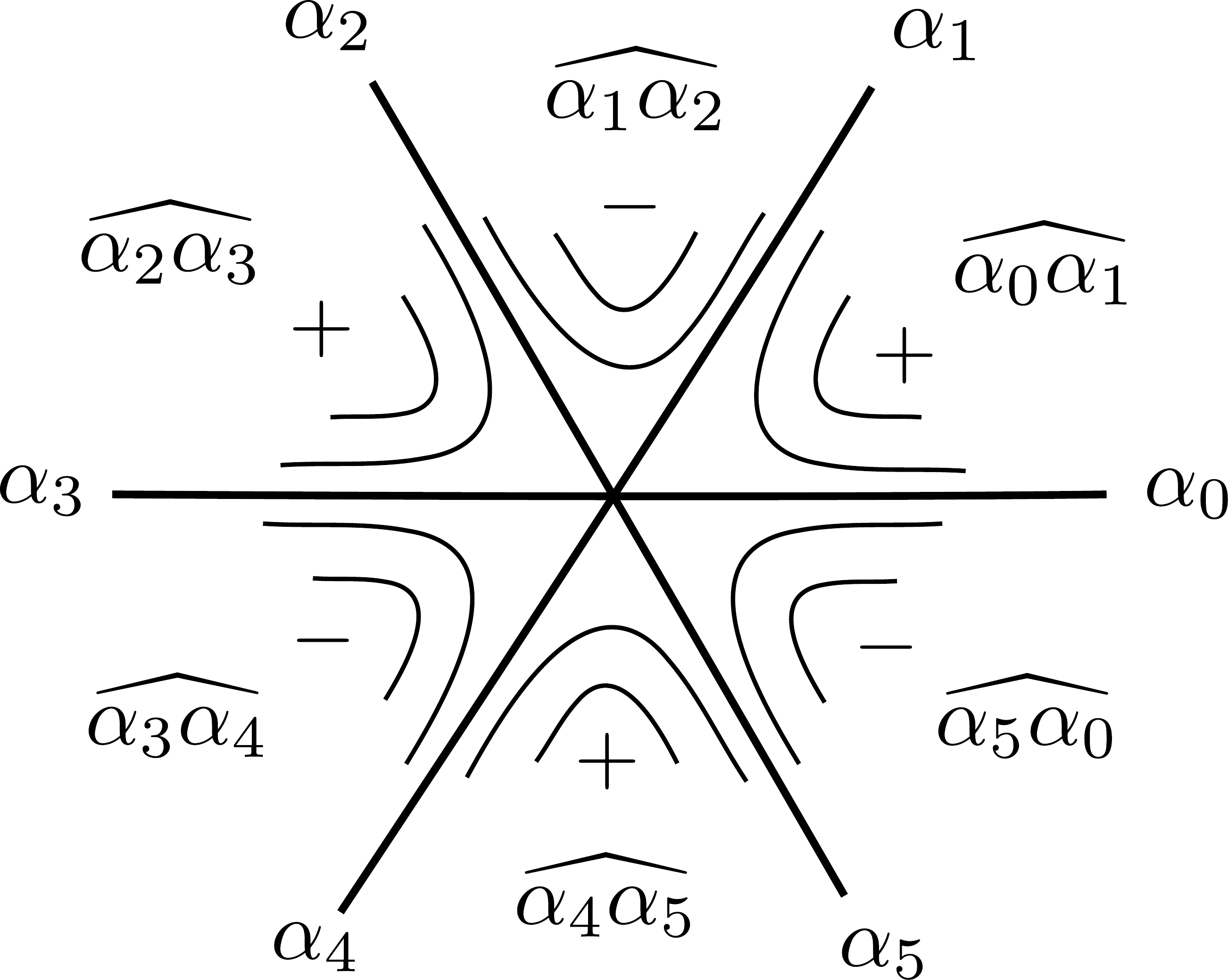}
\caption{Arcs and sectors near saddle point, $k=3$}
\label{fig:saddle_sectors}
\end{figure}

\subsection{Functions with isolated critical points on compact surfaces.}
From now on $\Mman$ will be a compact two-dimensional manifold.
Let $\func:\Mman\to\bR$ be a smooth ($\Cinfty$) function, $\fSing$ be the set of all critical points of $\func$, and $\Kman$ be a connected component of some level set $\func^{-1}(c)$, $c \in \bR$, of $\func$.
Then $\Kman$ is called \myemph{regular} if it contains no critical points of $\func$ and \myemph{critical}, otherwise.

\begin{definition}\label{def:classZ}
We will say that $\func$ belongs to \myemph{class $\FClass$} if
\begin{enumerate}[label={\rm\arabic*)}]
    \item\label{enum:classZ:bd} $\func$ takes constant values on the connected components of the boundary $\partial\Mman$;
    \item\label{enum:classZ:isol} each critical point of $\func$ is isolated and is contained in $\Int{\Mman}$.
\end{enumerate}
\end{definition}
In particular, every Morse function belongs to $\FClass$.

Suppose $\func\in\FClass$.
Then every connected component $\Kman$ of some level set of $\func$ have the following structure.
\begin{enumerate}[label={\rm\Alph*)}, leftmargin=*]
\item\label{enum:reg_comp}
Suppose $\Kman$ is regular, so it is a closed connected $1$-submanifold of $\Int{\Mman}$.
Therefore $\Kman$ is diffeomorphic with the circle $S^1$.
Moreover, there exist an open neighborhood $\Uman$ of $\Kman$, $\varepsilon>0$, and a diffeomorphism $\phi: S^1\times (-\varepsilon, \varepsilon) \to \Uman$ such that
$\phi(S^1\times0) = \Kman$ and $\func\circ \phi(z,t) = t + c$ for all $(z,t)\in S^1\times (-\varepsilon, \varepsilon)$.

\item\label{enum:crit_comp}
If $\Kman$ is critical, then it follows from~\ref{enum:classZ:isol} that $\Kman$ is homeomorphic to a finite $1$-dimensional CW-complex (``topological graph'').
Moreover, let $N$ be a connected component of $\func^{-1}[c-\varepsilon, c+\varepsilon]$ containing $\Kman$, where $\varepsilon>0$ is so small that $N \cap \partial\Mman=\varnothing$ and $N \cap \fSing = \Kman \cap \fSing$.
We will call $N$ an \myemph{$\func$-regular} neighborhood of $\Kman$, or an \myemph{atom} in the sense of A. Fomenko, see e.g.~\cite{BolsinovFomenko:1998}.

Let also $V$ be a connected component of $N\setminus \Kman$.
Then there exists a continuous map $\phi: S^1\times[-1,1] \to N$ with the following properties:
\begin{enumerate}[label=$\bullet$]
\item the set $F :=\phi^{-1}(\fSing)$ is finite and is contained in $S^1\times 1$;
\item the restriction of $\phi$ to $(S^1\times[-1,1]) \setminus F$ is an embedding;
\item $\phi\bigl( S^1\times[-1,1) \bigr) = V$;
\item $\phi$ homeomorphically maps each connected component $J$ of $(S^1\times 1) \setminus F$ onto some edge of $\Kman$.
\end{enumerate}
Thus, saying informally, $N$ can be obtained from $\Kman$ by gluing to it cylinders $S^1\times[-1,1]$ via maps $\psi:S^1\times 1 \to \Kman$ being homeomorphisms expect some finite subsets, see Figure~\ref{fig:nbh}.
\end{enumerate}

\begin{figure}[ht]
\centering\includegraphics[height=3.5cm]{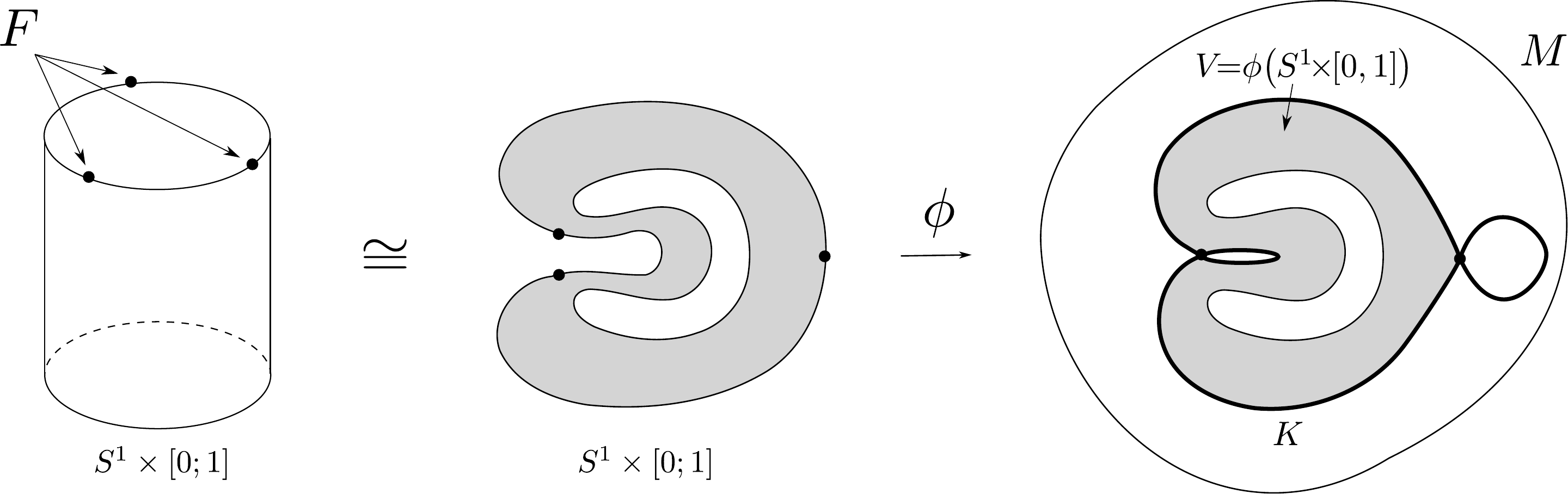}
\caption{Gluing cylinders to critical component of level set}
\label{fig:nbh}
\end{figure}

\subsection{Kronrod-Reeb graph}\label{sect:KRGraph}
For each \myemph{continuous} function $\func\in\Cont{\Mman}{\bR}$ we will denote by $\KRGraphf$ the partition of $\Mman$ into connected components of the level set of $\func$.
Let also $p:\Mman \to \KRGraphf$ be canonical quotient map associating to each point $x \in \Mman$ the connected component of the level set $\func^{-1}(\func(x))$ containing $x$.
Endow $\KRGraphf$ with the \myemph{factor} topology with respect to the mapping $p$, so a subset $A\subset \KRGraphf$ is open if and only if its inverse image $p^{-1}(A)$ is open in $\Mman$.
Then $\func$ induces a unique continuous function $\hat{\func}:\KRGraphf \to \bR$, such that $\func=\hat{\func}\circ p$.

It follows from~\ref{enum:reg_comp} and~\ref{enum:crit_comp} above that for $\func\in\FClass$ the space $\KRGraphf$ has a structure of a one-dimensional CW-complex: the vertices of $\KRGraphf$ correspond to critical components of level-sets of $\func$, while points of edges correspond to regular ones.
The space $\KRGraphf$ is often called \myemph{Kronrod-Reeb} graph, or \myemph{Lyapunov} graph, or simply the \myemph{graph} of $\func$, \cite{AdelsonWelskyKronrod:DANSSSR:1945, Kronrod:UMN:1950, Reeb:ASI:1952, Franks:Top:1985}.

\subsection{Action of the stabilizers of $\func$ on $\KRGraphf$.}
Notice that for each subgroup $\GRP$ of the group $\Homeo(\Mman)$ of homeomorphisms of $\Mman$ one can define a natural action $\phi:\Cont{\Mman}{\bR}\times\GRP \to \Cont{\Mman}{\bR}$ of $\GRP$ on the space $\Cont{\Mman}{\bR}$ of continuous functions on $\Mman$ defined by
\[
    \phi(\func,\dif) = \func\circ\dif:\Mman\to\bR.
\]

Given $\func\in\Cont{\Mman}{\bR}$ we will denote by 
\[
    \GStabilizer{\func} = \{ \dif\in\GRP  \mid \func\circ\dif = \func \}
\]
its stabilizer with respect to the above action.
Notice the relation $\func\circ\dif = \func$ means that $\dif(\func^{-1}(c))=\func^{-1}(c)$ for all $c\in\bR$, that is $\dif$ leaves invariant each level-set $\func^{-1}(c)$ of $\func$.
Hence it interchanges connected components of $\func^{-1}(c)$ and therefore induces a map $\rho(\dif):\KRGraphf \to \KRGraphf$ making commutative the following diagram, see Figure~\ref{fig:graphkr}:
\begin{equation}\label{equ:2x2_M_Graph}
\aligned
\xymatrix{
\Mman \ar[rr]^-{p} \ar[d]_-{\dif} &&
\KRGraphf \ar[rr]^-{\hat{\func}} \ar[d]^-{\rho(\dif)} &&
\bR \ar@{=}[d]  \\
\Mman \ar[rr]^-{p} &&
\KRGraphf \ar[rr]^-{\hat{\func}} &&
\bR
}
\endaligned
\end{equation}

\begin{figure}[ht]
\centering\includegraphics[height=3cm]{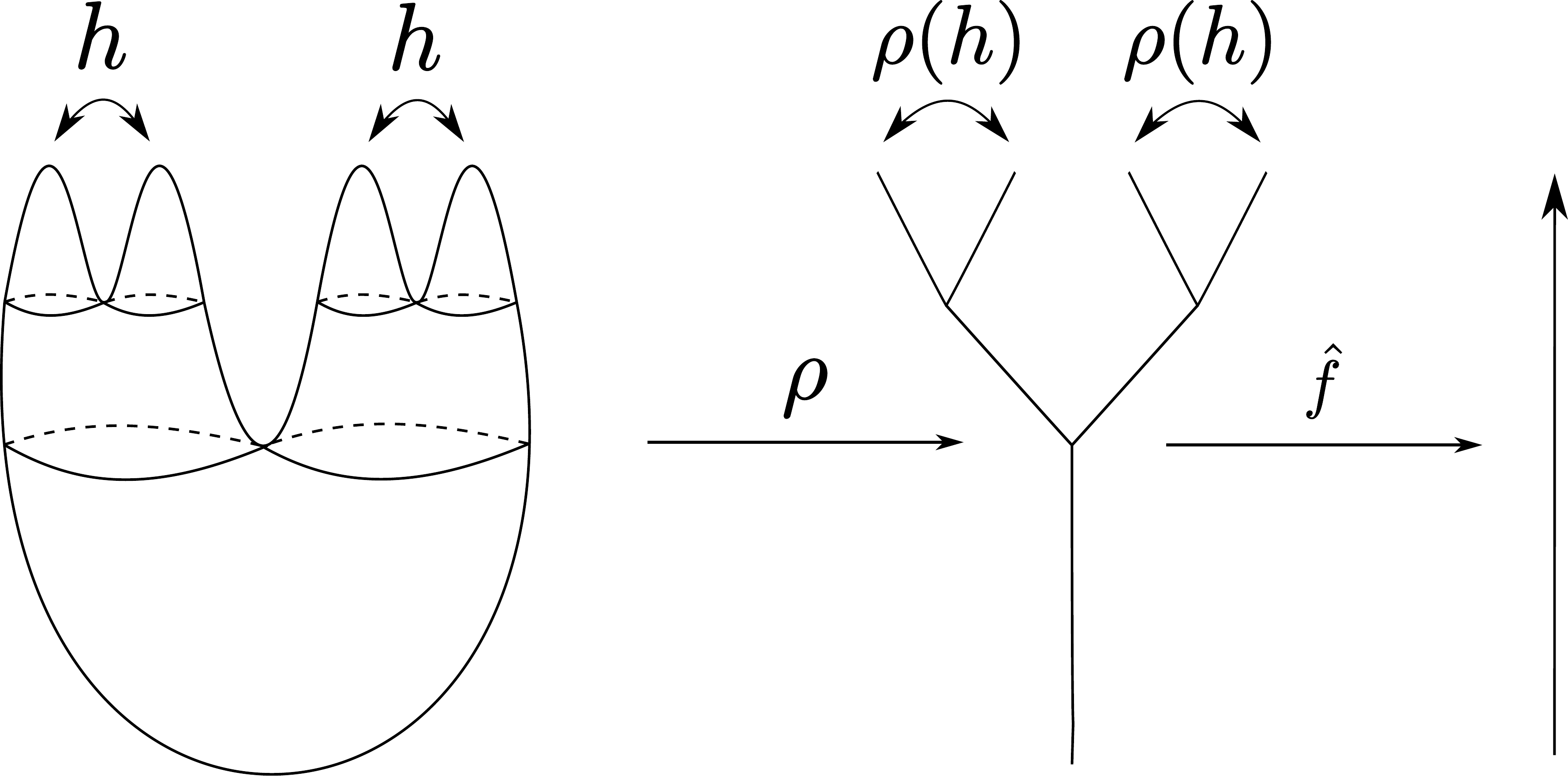}
\caption{Action of $\GStabilizer{\func}$ on $\KRGraphf$}
\label{fig:graphkr}
\end{figure}
    
Denote by $\Homeo(\KRGraphf)$ the group of homeomorphisms of $\KRGraphf$.
Then the following Lemma~\ref{lm:induced_map} implies that $\rho(\dif)$ is a \myemph{homeomorphism} of $\KRGraphf$, and one easily checks that the correspondence
\begin{equation}\label{equ:rho}
    \rho:\GStabilizer{\func} \to \Homeo(\KRGraphf)
\end{equation}
is a \myemph{homomorphism of groups}.
In other words, $\GStabilizer{\func}$ acts on $\KRGraphf$.

\begin{lemma}\label{lm:induced_map}
Suppose we have the following commutative diagram:
\[
\xymatrix
{
\Mman \ar[d]_-{\dif} \ar[r]^-{p} &  Y \ar[d]^-{g} \\
\Mman \ar[r]^-{p}  &  Y
}
\]
in which $\Mman$ and $Y$ are topological spaces, $\dif$ is continuous, and $p$ is a surjective factor map.
Then $g$ is continuous as well.
Moreover, if $\dif$ is a homeomorphism, then so is $g$.
\end{lemma}
\begin{proof}
We should show that $g^{-1}(U)$ is open in $Y$ for each open $U\subset Y$.
Since $\dif$ and $p$ are continuous, we see that
\[ p^{-1}\bigl(g^{-1}(U)\bigr) = \dif^{-1}\bigl( p^{-1}(U) \bigr)\]
is open in $\Mman$.
But $p$ is a factor map, whence an openness of the inverse image $p^{-1}\bigl(g^{-1}(U)\bigr)$, implies that $g^{-1}(U)$ is open in $Y$.
\end{proof}

Assume now that $\func\in\FClass$ and denote
\[
    \fG = \rho(\GStabilizer{\func}),
\]
so $\fG$ is the group of all homeomorphisms of $\KRGraphf$ induced by some homeomorphism $\GRP$ preserving $\func$, i.e. beloning to $\GStabilizer{\func}$.

One easily checks that if $\dif(e)=e$ for some edge $e$ of $\KRGraphf$ and $\dif\in\GStabilizer{\func}$, then $\dif|_{e} = \id_{e}$.
This implies that $\fG$ is a finite subgroup of $\Homeo(\KRGraphf)$ and can be regarded as a group of certain \myemph{automorphisms} of a ``graph'' $\KRGraphf$.

\subsection{Functions on $2$-sphere}
Suppose now $\Mman = \Sp$ and $\func\in\FClass$.
Then $\KRGraphf$ is always a tree.
We claim that the set $\Fix{\fG}$ of common fixed points of $\fG$ is a non empty subtree of $\KRGraphf$.

Indeed, it is well known that the group of automorphisms of a finite tree always has
\begin{itemize}[]
\item
either a common fixed point or
\item
an invariant edge $e$ such that some automorphisms of $\KRGraphf$ change orientation of $e$.
\end{itemize}
However, as mentioned above, elements of $\fG$ do not change orientation of edges, whence $\fG$ must have fixed points.
Moreover, if $v,w$ are two fixed points of $\fG$, then there exists a unique path $\pi$ in $\KRGraphf$ connecting them, whence this path is $\fG$-invariant, and hence it must consist of fixed ponts of $\fG$ as well.
Hence $\Fix{\fG}$ is a non-empty subtree of $\KRGraphf$.

For a vertex $v \in \Fix{\fG}$ let $Star(v)$ be the \myemph{star} of $v$ in $\KRGraphf$, that is the union of $v$ and all edges incident to $v$.
Then $Star(v)$ is invariant with respect to $\fG$, whence we can define the group
\[\lG=\{\phi|_{Star(v)} \mid \phi \in \fG\}\]
consisting of the restrictions of automorphisms of $\fG$ to $Star(v)$.
Notice that $\fG$ is can be regarded as a subgroup of the permutation group of edges of $Star(v)$.
We will call $\lG$ the \myemph{local stabilizer} of $v$, see Remark~\ref{rem:Gvloc}.

In particular, we have an epimorphism
\begin{align*}
&r_{v}:\fG \to \lG, & r_{v}(\phi)=\phi_{|Star(v)},
\end{align*}
and the composition
\begin{equation}\label{eq:rv}
\rho_{v}\colon \  \GStabilizer{\func} \ \stackrel{\rho}{\longrightarrow} \ \fG \ \stackrel{r_{v}}{\longrightarrow} \ \lG.
\end{equation}

Our aim is to prove the following statement.
\begin{theorem}\label{th:cw3}
Let $\func\in\Cont{\Sp}{\bR}$ be a function from class $\FClass$ on $2$-sphere $\Sp$.
Let also $\GRP \subset \Homeo^{+}(\Sp)$ be any subgroup of the group of orientation preserving homeomorphisms of $\Sp$, $\fG = \rho(\GStabilizer{\func})$, and $v\in\Fix{\fG}$ be a common fixed vertex of the group $\fG$.
Then the following statement holds.
\begin{enumerate}[label={\rm(\alph*)}, wide, itemsep=1ex]
\item\label{enum:gloc:Gvloc_in_SO3}
$\lG$ is isomorphic to a finite subgroup of $\Homeo^{+}(\Sp)$.
Therefore, by Brouwer-Kerekjarto theorem~\cite{MR1309126} $\lG$ is isomorphic to a finite subgroup of $SO(3)$, and thus to one of the following groups:
\begin{equation}\label{eq:mh1}
	\bZ_{n},  \    \ \bD_{n},  \    \ \bA_{4}, \    \ \bS_{4},  \   \ \bA_{5},  \   \ n\geq1,
\end{equation}
see e.g.~{\rm\cite[pages 21-23]{Klein:1888}}.

\item\label{enum:gloc:FixGf_edge}
If $\Fix{\fG}$ has at least one edge, then for any vertex $v\in\Fix{\fG}$, the group $\lG$ is cyclic.

\item\label{enum:gloc:FixGf_vertex}
Suppose $\Fix{\fG}$ consists of a unique vertex $v$ and $\lG \cong \bZ_k$ for some $k\geq2$.
Let also $\XX$ be the critial component of a level set of $\func$ corresponding to $v$.
Then there are two critical saddle points $z_1, z_2\in \XX$ of orders $k_1$ and $k_2$ respectively, such that $k$ divides $GCD(k_1,k_2)$ and every $\dif\in\GStabilizer{\func}$ fixes $z_1$ and $z_2$.
\end{enumerate}
\end{theorem}

Theorem~\ref{th:cw3} will be deduced in \S\ref{sect:proof:th:cw3:a} and \S\ref{sect:proof:th:cw3:bc} from Theorems~\ref{th:cw1} and~\ref{th:cw2} below about cellular homeomorphisms of CW-complexes.

\subsection{Proof of Theorem~\ref{th:main}}\label{sect:proof:th:main}
We need only to establish statement~\ref{enum:mm3} of Theorem~\ref{th:main}.
Notice that each Morse function $\func:\Sp\to\bR$ has isolated critical points, whence Theorem~\ref{th:cw3} is applicable for $\func$.
In particular, statements~\ref{enum:m1} and~\ref{enum:m3} of Theorem~\ref{th:main} are the same as the corresponding statements in Theorem~\ref{th:cw3}.

\ref{enum:m4}
Suppose $\Fix{\fG}$ consists of a unique point and $\lG \cong \bZ_k$ for some $k\geq2$.
We need to show that $k=2$.
Indeed, by~\ref{enum:gloc:FixGf_vertex} of Theorem~\ref{th:cw3} $k$ must divide the order of some saddle critical point of $\func$.
But each \myemph{non-degenerate} saddle has order $2$, whence $k=2$.

\section{Celluar homeomorphisms of CW-complexes}\label{sect:cw-partitions}
In what follows $D^{\kk}$, $\kk\geqslant 1$, is the closed $k$-disk of radius one in $\bR^{\kk}$ with center at the origin.

Let $X$ be a CW-complex.
Then by $\Xk$, $\kk\geqslant0$, we will denote its $\kk$-skeleton and by $j_{\kk}:\Xk\hookrightarrow X$ the canonical embedding.
A \myemph{cell} always mean an \myemph{open} cell.
Also $0$- and $1$-cells will often be called \myemph{vertices} and \myemph{edges} respectively.

A homeomorphism $\dif: X\to X$ is \myemph{cellular} if $\dif(\Xk)=\Xk$ for all $\kk\geq0$, i.e. for every cell $e$ of $X$ its image, $\dif(e)$, is also a cell of $X$.
We will denote by $\Homeo^{\cw}(X)$ the group of all cellular homeomorphisms of $X$.

Note that for each $\kk\geq0$ there is a \myemph{restriction to $\Xk$} homomorphism:
\begin{align*}
&\rho_{\kk}:\Homeo^{\cw}(X) \to \Homeo^{\cw}(\Xk), &
\rho_{\kk}(\dif)=\dif|_{\Xk}.
\end{align*}

Let also
\[\Homeo^{\cw}(\Xk,j_{\kk}) := \rho_{\kk}\bigl(\Homeo^{\cw}(X)\bigr)\]
be the image of $\rho_{\kk}$.
Evidently, it consists of cellular homeomorphisms $\dif$ of $\Xk$, which can be extended to some cellular homeomorphism of $X$.
In particular, we have an \myemph{epimorphism}
\[ \rho_{k}:\Homeo^{\cw}(X)\to \Homeo^{\cw}(\Xk,j_{\kk}).\]

Finally, let $\Homeo_{0}^{\cw}(X) \subset \Homeo^{\cw}(X)$ be the subgroup consisting of homeomorphisms which leave invariant each cell $e$ of $X$ and preserve its orientation if $\dim e \geq 1$.

\begin{theorem}\label{th:cw1}
Let $\XX$ be a $1$-dimensional CW complex, and $\metr$ be a metric on $\XX$ such that the length of each edge equals $1$.
Let also $\Is(\XX)\subset\Homeo^{\cw}(\XX)$ be the subgroup consisting of all cellular isometries of $\XX$.
Then the following statements hold.
\begin{enumerate}[label={\rm(\arabic*)}, wide, itemsep=2ex]
\item\label{enum:th:cw1:homo}
There exists a homomorphism
\begin{align*}
q:\Homeo^{\cw}(\XX) \to \Is(\XX),
\end{align*}
being a retraction onto $\Is(\XX)$, that is, $q(\dif)=\dif$ for each $\dif \in \Is(X)$.

\item\label{enum:th:cw1:ker}
$\ker(q) = \Homeo_{0}^{\cw}(\XX)$, so we get the following exact sequence:
\begin{equation}\label{equ:tp1}
\xymatrix{
    1 \ar[r] & \Homeo_{0}^{\cw}(\XX) \ar[r] &
    \Homeo^{\cw}(\XX) \ar@/^/[r]^-{q}
     & \Is(\XX) \ar@/^/[l]^-{\eta} \ar[r] & 1
}
\end{equation}
in which the natural inclusion $\eta:\Is(\XX)\subset\Homeo^{\cw}(\XX)$ is a section of $q$.

Thus, we have a splitting of $\Homeo^{\cw}(\XX)$ into a semidirect product of its subgroups:
\[ \Homeo^{\cw}(\XX)=\Homeo_{0}^{\cw}(\XX)\rtimes \Is(\XX).\]

In partcular, for each $\dif \in \Homeo^{\cw}_{0}(\XX)$ its image $\{q(\dif)\}$ is the only element of the intersection of its adjacent class $h\Homeo^{\cw}_{0}(\XX)$ with $\Is(\XX)$:
\begin{equation}\label{equ:tpq}
\{q(\dif)\} \ = \ h\Homeo^{\cw}_{0}(\XX)\ \cap \ \Is(\XX).
\end{equation}
\end{enumerate}
\end{theorem}
\begin{proof}
\ref{enum:th:cw1:homo}
Let $\{e_{i}\}_{i\in\Lambda}$ be all $1$-cells of $\XX$.
For each $i\in\Lambda$ let
\begin{align*}
\phi_{i}:[-1,1]\to \XX,
\end{align*}
be the characteristic map of $e_{i}$, so the restriction $\phi_{i}|_{\{-1,1\}}:\{-1,1\}\to X^{0}$ of $\phi_{i}$ to the boundary $\partial [-1,1] = \{-1,1\}$ is the \myemph{gluing map} of the cell $e_{i}$, and
\begin{align*}
\phi_{i}|_{(-1,1)}:(-1,1) \to \XX\setminus X^{0}
\end{align*}
is an embedding.

By assumption, the length of each $1$-cell in the metric $\metr$ equals $1$, so without loss of generality, we may assume that the restriction
\begin{align*}
\phi_{i}|_{(-1,1)}:(-1,1)\to e_{i}
\end{align*}
is an isometry.

Let $\dif \in \Homeo^{\cw}(\XX)$.
We need to construct an isometry $q(\dif):\XX\to \XX$ so that the correspondence $q(\dif)\mapsto\dif$ will satisfy the assertion of the theorem.

We will define $q(\dif)$ in such a way that:
\begin{enumerate}[label=(\alph*), itemsep=1ex]
\item\label{enum:a} $q(\dif)(v)=\dif(v)$ for each vertex $v\in X^{0}$;
\item\label{enum:b} $q(\dif)(e)=\dif(e)$ for each edge $e$ of $\XX$;
\item\label{enum:c} if $\dif(e)=e$, then $\dif$ preserves the orientation of $e$, if and only if $q(\dif)|_{e}=\id_{e}$.
\end{enumerate}

According to~\ref{enum:a} we must put $q(\dif)(v)=\dif(v)$ for all $v\in X^{0}$, so it remains to extend $q(\dif)$ to all of $\XX$.

Let $e_{i}$ be an edge of $\XX$ and $e_{j}=\dif(e_{i})$ its image is under $\dif$.
Then there is a unique homeomorphism $\alpha_{i}:[-1,1]\to[-1,1]$ such that the following diagram is commutative:
\begin{equation}
\aligned
    \xymatrix{
	(-1,1) \ar[rr]^-{\alpha_{i}} \ar[d]_-{\phi_{i}} && (-1,1) \ar[d]^-{\phi_{j}}\\
	e_{i}\ar[rr]^-{\dif} && e_{j}
}
\endaligned
\end{equation}
that is $\alpha_{i}|_{(-1,1)}=\phi_{j}\circ\dif\circ{\phi_{i}}^{-1}$.

Due to~\ref{enum:b} define $q(\dif)|_{e_i}: e_{i} \to \dif(e_i) = e_j$ by
\begin{equation*}
q(\dif)|_{e_{i}}=
\begin{cases}
\phi_{j}\circ \id_{(-1,1)}\circ{\phi_{i}}^{-1},  & \text{if $\alpha_{i}$ preserves the orientation},\\
\phi_{j}\circ{(-\id_{(-1,1)})}\circ{\phi_{i}}^{-1}, & \text{if $\alpha_{i}$ reverses the orientation},
\end{cases}
\end{equation*}
where $\id_{(-1,1)}$ is the identity map of $(-1,1)$, and $-\id_{(-1,1)}(t) = -t$ for all $t\in(-1,1)$.
In other words, we get the following commutative diagram:
\begin{equation}
\aligned
\xymatrix{
	(-1,1) \ar[rr]^-{\pm \id_{(-1,1)}} \ar[d]_-{\phi_{i}} && (-1,1) \ar[d]_-{\phi_{j}}\\
	e_{i}\ar[rr]^-{q(\dif)}&& e_{j}
}
\endaligned
\end{equation}
Then~\ref{enum:c} also holds.
Moreover, since $\pm \id_{(-1,1)}$ and $\phi_{i}|_{(-1,1)}$ for all $i$ are isometries, we see that $q(\dif)|_{e_{i}}$ is an isometry as well.
Thus $q(\Homeo^{\cw}(\XX))\subset \Is(\XX)$.

Note that if $\dif \in \Is(\XX)$, then for each cell $e_{i}$ the homeomorphism
\begin{align*}
\alpha_{i}=\phi_{j}\circ\dif\circ{\phi_{i}}^{-1}
\end{align*}
is an isometry of the segment $[-1,1]$.
Hence $\alpha|_{(-1,1)}=\pm \id_{(-1,1)}$ and therefore
\begin{align*}
q(\dif)|_{e_{i}}=\phi_{j}\circ\alpha_{i}\circ{\phi_{i}}^{-1}=\dif|_{e_{i}}.
\end{align*}
In other words, $q(\dif)=\dif$ for each $\dif \in \Is(\XX)$.
Thus $q$ is a retraction $\Homeo^{\cw}(\XX)$ by $\Is(\XX)$.

Verification that $q$ is a homomorphism we leave for the reader.

\medskip
\ref{enum:th:cw1:ker}
The identity $\ker{q}=\Homeo_{0}^{\cw}(\XX)$ is a direct consequence of \ref{enum:a}-\ref{enum:c}.
All other statements are easy and we also leave them for the reader.
\end{proof}

\begin{theorem}\label{th:cw2}
Let $\Xkk$, $k\geq0$, be a $(\kk+1)$-dimensional CW complex.
Suppose that for each $(\kk+1)$-cell of $e$, it gluing map $\psi_{e}:S^{\kk}\to \Xk$ has the following property: there exists a (possibly empty) finite subset $F_{e}\subset S^{\kk}$ such that:
\begin{enumerate}[label={\rm(\alph*)}, itemsep=1ex, wide]
\item\label{enum:cw2:assump:a}
${\psi_{e}}^{-1}(X^{0})=F_{e}$;

\item\label{enum:cw2:assump:b}
$\psi_{e}|_{S^{\kk}\setminus F_{e}}:S^{\kk}\setminus F_{e}\to \Xk\setminus X^{0}$ is an embedding, i.e.\! a homeomorphism on its image.
\end{enumerate}
Then the following statemens hold true.
\begin{enumerate}[label={\rm\arabic*)}, itemsep=1ex, leftmargin=*]
\item\label{enum:cw2:sect}
There exists a homomorphism $s:\Homeo^{\cw}(\Xk,j_{k})\to\Homeo^{\cw}(\Xkk)$ being a section of $\rho_{k}:\Homeo^{\cw}(\Xkk)\to \Homeo^{\cw}(\Xk,j_{\kk})$, i.e. $\rho_{k}(s(\dif))=s(\dif)|_{\Xk}=\dif$, for all $\dif \in \Homeo^{\cw}(\Xk,j_{\kk})$.

\item\label{enum:cw2:2dim}
Suppose $\kk=1$, so $\dim X^{2}=2$.
Then
\begin{enumerate}[label={\rm(\roman*)}, itemsep=1ex, topsep=1ex]
\item\label{enum:cw2:2dim:H0}
$\Homeo^{\cw}_{0}(\XX)\subset\Homeo^{\cw}(\XX,j_{1})$.

\item\label{enum:cw2:2dim:qH-IsX1}
Fix a metric $\metr$ on $\XX$ in which each $1$-cell has length $1$ and let \[q:\Homeo^{\cw}(\XX) \to \Is(\XX)\] be the homomorphism constructed in Theorem~\ref{th:cw1}.
Let also
\begin{equation}\label{equ:is1}
\Is(\XX,j_{1})=\Homeo^{\cw}(\XX,j_{1})\cap\Is(\XX).
\end{equation}
Then
\begin{equation}\label{equ:is2}
q(\Homeo^{\cw}(\XX,j_{1}))=\Is(\XX,j_{1}).
\end{equation}
In other words, if a cellular homeomorphism $\dif: \XX\to\XX$ extends to a cellular homeomorphism of $X^2$, then the corresponding isometry $q(\dif)$ also extends to a cellular homeomorphism $X^2$.

In particular, we have the following commutative diagram in which the first line coincides with~\eqref{equ:tp1}:
\begin{equation}\label{equ:tp2}
\aligned
\xymatrix{
1 \ar[r]&\Homeo_{0}^{\cw}(\XX) \ar[r] \ar@{=}[d]&\Homeo^{\cw}(\XX) \ar@/^/[rr]^-{q} &&
\Is(\XX)\ar@/^/[ll]^{\eta}  \ar[r]&1 \\
1 \ar[r]&\Homeo_{0}^{\cw}(\XX) \ar[r] &
\Homeo^{\cw}(\XX,j_{1}) \ar@/^/[rr]^-{q} \ar@{^{(}->}[u] &&
\Is(\XX,j_{1})\ar@/^/[ll]^{\eta} \ar[r] \ar@{^{(}->}[u]&1
}
\endaligned
\end{equation}
\end{enumerate}
\end{enumerate}
\end{theorem}
\begin{proof}
\ref{enum:cw2:sect}
Let $\dif \in \Homeo^{\cw}(\Xk,j_{\kk})=\rho_{\kk}(\Homeo^{\cw}(\Xkk))$, so
\begin{align*}
	\rho(\hat{\dif})=\hat{\dif}|_{\Xk}=\dif.
\end{align*}
for some $\hat{\dif} \in \Homeo^{\cw}(\Xkk)$.
In other words $\dif$ can be extended to some cellular homeomorphism $\hat{\dif}$ of $\Xkk$, thought this extension is not unique.

We need to construct an extensions $\hat{\dif}=s(\dif)$ of all $\dif\in\Homeo^{\cw}(\Xk,j_{\kk})$ so that the correspondence of $s:\dif \to \hat{\dif}$ will be a homomorphism of groups.
In fact, one need to define $s(\dif)$ only for each $(\kk+1)$-cell.

Let $\psi_{e}:D^{\kk+1}\to \Xkk$ be characteristic mapping of a $(k+1)$-cell $e$.
Then by~\ref{enum:cw2:assump:a} and~\ref{enum:cw2:assump:b} there exists a finite subset $F_{e}\subset \partial{D^{k+1}}=S^{\kk}$ such that the gluing map
\begin{align*}
	\psi_{e}|_{S^{\kk}\setminus F_{e}}\colon \ S^{\kk}\setminus F_{e} \ \to \ \Xk\setminus X^{0},
\end{align*}
is an embedding.
One easily checks that then the restriction
\begin{align*}
	\psi_{e}|_{D^{\kk+1}\setminus F_{e}}\colon \ D^{\kk+1}\setminus F_{e} \ \to \ \Xk\setminus X^{0}
\end{align*}
will also be an embedding.

Let $e' = \hat{\dif}(e)$ be the image of $e$.
Then the homeomorphism
\begin{align*}
    \alpha_{e}=\psi^{-1}_{e{'}}\circ\hat{\dif}\circ\psi_{e} \colon
        \ D^{\kk+1}\setminus{F_{e}} \ \to \ D^{\kk+1}\setminus{F_{e'}}
\end{align*}
makes the following diagram commutative:
\begin{equation}\label{equ:3}
	\aligned
	\xymatrix{
		D^{k+1}\setminus F_{e}\ar[rr]^-{\psi_{e}} \ar[d]_-{\alpha_{e}} && \Xkk \ar@{<-^)}[rr]^-{j_{k}} \ar[d]^-{\hat{\dif}} && \Xk\ar[d]^-{\dif}\\
		D^{k+1}\setminus F_{e'}\ar[rr]^-{\psi_{e'}} && \Xkk \ar@{<-^)}[rr]^-{j_{k}} && \Xk
	}
	\endaligned
\end{equation}
Since $F_{e}$ and $F_{e'}$ are finite subsets of $\partial{D^{\kk+1}}$, they must consist of the same number of points.
Therefore $\alpha_{e}$ extends to a unique homeomorphism $\alpha_{e}:D^{k+1}\to D^{\kk+1}$ making commutative the following diagram:
\begin{equation}\label{equ:2x2}
	\aligned
	\xymatrix{
		\partial{D^{k+1}} \ar[rr]^-{\psi_{e}} \ar[d]_-{\alpha|_{D^{k+1}}} &&  \Xk \ar[d]^-{\dif}\\
		\partial{D^{k+1}} \ar[rr]^-{\psi_{e'}}&&  \Xk
	}
	\endaligned
\end{equation}
Define another homeomorphism $\beta(\dif): D^{\kk+1}\to D^{\kk+1}$ by
\begin{equation}\label{equ:beta}
	\beta_{e}(\dif)(x)=
	\begin{cases}
	0, & x=0,\\
	\alpha{(x/|x|)}|x|, & x\not=0.
	\end{cases}
\end{equation}
Then $\beta(\dif) = \alpha_{e}$ on $\partial{D^{\kk+1}}$.
Extend $\dif$ to a homeomorphism $s(\dif):\Xkk\to \Xkk$ by:
\begin{equation}\label{equ:s}
s(\dif)=
\begin{cases}
  \dif(x), & x\in \Xk,\\[1mm]
  \psi_{e'}\circ\beta_{e}(\dif)\circ\psi_{e}^{-1}, & x\in e \subset X^{k+1} \setminus \Xk.
\end{cases}
\end{equation}
Then $s(\dif)$ is a cellular homeomorphism, i.e.\! it belongs to $\Homeo^{\cw}(\Xkk)$, and for each $(\kk+1)$-cell $e$ it makes the following diagram commutative:
\begin{equation}\label{equ:3x2}
\aligned
\xymatrix{
    D^{k+1}\ar[rr]^-{\psi_{e}} \ar[d]_-{\beta_{e}(\dif)} && \Xkk \ar@{<-^)}[rr]^-{j_{k}} \ar[d]^-{s(\dif)} && \Xk\ar[d]^-{\dif}\\
    D^{k+1}\ar[rr]^-{\psi_{e'}} && \Xkk \ar@{<-^)}[rr]^-{j_{k}} && \Xk
}
\endaligned
\end{equation}
It is easy to verify that the correspondence $\dif \to s(\dif)$ is a homomorphism of groups
\begin{align*}
	s:\Homeo^{\cw}(\Xk,j_{\kk})\to\Homeo^{\cw}(\Xk),
\end{align*}
being also a section of $\rho_{\kk}$, i.e. $s(\dif)|_{\Xk}=\dif$ for all $\dif \in \Homeo^{\cw}(\Xk,j_{\kk})$.

\medskip

\ref{enum:cw2:2dim}
Now consider the case $\dim X^{2}=2$.

\ref{enum:cw2:2dim:H0}
We need to show that $\Homeo^{\cw}_{0}(\XX)\subset\Homeo^{\cw}(\XX,j_{1})$, i.e., that every homeomorphism $\dif \in \Homeo^{\cw}_{0}(\XX)$ can be extended to some homeomorphism
\begin{align*}
	\gamma(\dif)&:X^{2} \to X^{2}.
\end{align*}

It suffices to show that for each $2$-cell $e$ there exists a homeomorphism $\beta_{e}$ which makes the following diagram commutative:
\begin{equation}\label{equ:kd1}
	\aligned
	\xymatrix{
		D^{2} \ar[rr]^-{\psi_{e}} \ar[d]_-{\beta_{e}} && X^{2} \ar@{<-^)}[rr]^-{j_{1}} \ar@{-->}[d]^-{\gamma(\dif)} && \XX\ar[d]^-{\dif}\\
		D^{2} \ar[rr]^-{\psi_{e}} && X^{2} \ar@{<-^)}[rr]^-{j_{1}}  &&  \XX
	}
	\endaligned
\end{equation}
where $\psi_{e}:D^{2}\to \XX$ the characteristic mapping of the cell $e$.
Then $\gamma(\dif)$ will be uniquely determined by its restrictions on each $2$-cell $e$ by the formula:
\[
	\gamma(\dif)|_{e}=\psi^{-1}_{e}\circ\beta_{e}\circ\psi_{e}.
\]

To construct $\beta_{e}$, recall that there is a finite subset of $F_{e}\subset \partial{D^{2}}=\Cir$ such that the corresponding gluing map
\[
	\psi_{e}|_{{\Cir}\setminus F_{e}}:\Cir\setminus F_{e}\to \XX\setminus X^{0}
\]
is an embedding.
Therefore, there is a homeomorphism $\alpha_{e}:\Cir\setminus F_{e}\to \Cir\setminus F_{e}$ which makes commutative the following diagram:
\begin{equation*}
	\aligned
	\xymatrix{
		\Cir\setminus{F_{e}} \ar@{^{(}->}[r] \ar[d]^-{\alpha_{e}}&\Cir \ar@{^{(}->}[r]&D^{2} \ar[r]^-{\psi_{e}}&X^{2} \ar@{<-^)}[r]^-{j_{1}}&\XX\ar[d]^-{\dif}\\
		\Cir\setminus{F_{e}} \ar@{^{(}->}[r]&\Cir\ar@{^{(}->}[r]&D^{2} \ar[r]^-{\psi_{e}}&X^{2} \ar@{<-^)}[r]^-{j_{1}}&\XX
	}
	\endaligned
\end{equation*}

Notice that for an each connected component $K$ of the set $\Cir\setminus F_{e}$ its image $\psi_{e}(K)$ is an edge in $X^{2}$.
As $\dif \in \Homeo^{\cw}_{0}(\XX)$, it follows that $\dif$ leaves invariant that edge and preserve its orientation.
In addition, $\dif$ preserves fixed each vertex of $\XX$, and therefore $\alpha_{e}$ extends to a homeomorphism
\[
	\alpha_{e}:\Cir\to\Cir
\]
which is fixed on $F_{e}$ and makes the following diagram commutative:
\begin{equation}\label{equ:kd5b}
	\aligned
	\xymatrix{
		\Cir\setminus{F_{e}} \ar@{^{(}->}[r] \ar[d]^-{\alpha_{e}}&\Cir \ar@{^{(}->}[r] \ar[d]^-{\alpha_{e}}&D^{2} \ar[r]^-{\psi_{e}}&X^{2} \ar@{<-^)}[r]^-{j_{1}}&\XX\ar[d]^-{\dif}\\
		\Cir\setminus{F_{e}} \ar@{^{(}->}[r]&\Cir\ar@{^{(}->}[r]&D^{2} \ar[r]^-{\psi_{e}}&X^{2} \ar@{<-^)}[r]^-{j_{1}}&\XX
	}
	\endaligned.
\end{equation}
Using Alexander's trick extend $\alpha_{e}$ to some homeomorphism $\beta_{e}:D^{2}\to D^{2}$.
Then $\beta_{e}$ is a desired homeomorphism making the diagram~\eqref{equ:kd1} commutative.

\medskip

\ref{enum:cw2:2dim:qH-IsX1}
We need to verify that $q(\Homeo^{\cw}(\XX,j_{1}))=\Is(\XX,j_{1})$, see~\eqref{equ:is2}.
Since
\[
	\Homeo^{\cw}_{0}(\XX)\subset\Homeo^{\cw}(\XX,j_{1}),
\]
it follows that for each $\dif \in \Homeo^{\cw}(\XX,j_{1})$ we have, that
\begin{equation}\label{equ:qhi2}
	\dif\Homeo^{\cw}_{0}(\XX)\subset\Homeo^{\cw}(\XX,j_{1}).
\end{equation}
Hence
\begin{multline*}
	\{q(\dif)\}\stackrel{\eqref{equ:tpq}}{=}h\Homeo^{\cw}_{0}(\XX)\cap\Is(\XX)\stackrel{\eqref{equ:qhi2}}{\subset}\\
	\stackrel{\eqref{equ:qhi2}}{\subset}\Homeo^{\cw}(\XX,j_{1})\cap\Is(\XX)\stackrel{\eqref{equ:is1}}{=}\Is(\XX,j_{1}).
\end{multline*}
In other words
\begin{align*}
	q(\Homeo^{\cw}(\XX,j_{1}))\subset\Is(\XX,j_{1}).
\end{align*}

To check the inverse inclusion recall that the homomorphism $q$ from Theorem~\ref{th:cw1} is a retraction onto $\Is(\XX)$.
Therefore, for each
\[
	\dif\in\Is(\XX,j_{1})\subset\Homeo^{\cw}(\XX,j_{1})
\]
$q(\dif)=\dif$, whence
\[
	\Is(\XX,j_{1})=q(\Is(\XX,j_{1}))\subset q(\Homeo^{\cw}(\XX,j_{1})).
\]
Theorem~\ref{th:cw2} is proved.
\end{proof}

\section{Proof of~\ref{enum:gloc:Gvloc_in_SO3} of Theorem~\ref{th:cw3}}\label{sect:proof:th:cw3:a}
Let $\XX$ be the critical component of a level set of $\func$ corresponding to the vertex $v$, $X^{0}$ be the set of critical points of $\func$ belonging to $\XX$, and $j_{1}:\XX \hookrightarrow \Sp$ be the canonical embedding.

Then $\Sp$ has a structure of a $2$-dimensional CW-complex in which: $X^{0}$, $\XX$, and $X^{2}=\Sp$ are $0$-, $1$ and $2$-skeletons, respectively.
We will denote by $\Partit$ the corresponding CW-partition of $\Sp$. 

Fix any metric $\metr$ such that the length of each edge equals $1$ and denote by $\Is(\XX)$ the group of isometries of $\XX$.
Similarly to Theorem~\ref{th:cw2} consider the group
\begin{align*}
\Is(\XX, j_{1})=\Is(\XX)\cap \Homeo^{\cw}(\XX,j_{1})
\end{align*}
of isometries of $\XX$ which can be extended to some homeomorphisms of $\Sp$ with respect to this embedding $j_{1}$.
Then $\Is(\XX,j_{1})$ is a finite subgroup of $\Homeo^{\cw}(\XX,j_{1})$.

As $\func$ has only finitely many critical points, and they are isolated, one can assume that the gluing map $\psi_{e}:\partial D^1 \to X^{0}$ of each $2$-cell $e$ is a homeomorphism outside some finite subset $F_e \subset \partial D^1$.

Thus, this cellular partition of $\Sp$ satisfies conditions of Theorem~\ref{th:cw2}.
Therefore the group $\Homeo^{\cw}(\XX,j_{1})$ and thus its finite subgroup $\Is(\XX,j_{1})$ are isomorphic to some subgroups of $\Homeo(\Sp)$.

Hence, to prove the theorem it suffices to show that $\lG$ is isomorphic to a subgroup of $\Is(\XX,j)$.

Recall that there is an epimorphism~\eqref{eq:rv}:
\begin{align*}
    \rho_{v}:\StabilizerIsotId{\func} \ \stackrel{\rho}{\longrightarrow} \
    \fG \ \stackrel{r_{v}}{\longrightarrow} \
    \lG.
\end{align*}

On the other hand, we have the homomorphism
\begin{align*}
    \sigma:\GStabilizer{\func}\ \subset \
        \Homeo^{\cw}(\Sp) \ \stackrel{\omega}{\longrightarrow} \
        \Homeo^{\cw}(\XX,j_{1}) \ \stackrel{q}{\longrightarrow} \
        \Is(\XX,j_{1}),
\end{align*}
where $\omega(\dif)=\dif|_{\XX}$ is the restriction homomorphism on $\XX$, and $q$ the homomorphism defined in Theorem~\ref{th:cw1}.
We will show that
\begin{equation}\label{equ:kk}
\ker(\rho_{v})=\ker(\sigma).
\end{equation}
As $\rho_{v}$ is an epimorphism, there will exist a unique \myemph{monomorphism}
\begin{align*}
	\mu:\lG \hookrightarrow \Is(\XX,j),
\end{align*}
making the following diagram commutative:
\begin{equation}
	\xymatrix{
		& \StabilizerIsotId{\func} \ar[ld]_-{\rho_{v}} \ar[rd]^-{\sigma}&  \\
		\lG \ar@{^{(}->}[rr]_-{\mu} &  &\Is(\XX,j_{1})
	}
\end{equation}
In other words, we will get that $\lG$ is isomorphic to some subgroup of $\Is(\XX,j)$ and hence of $\Homeo(\Sp)$.

Consider the following conditions on $\dif \in \StabilizerIsotId{\func}$:
\begin{center}
\begin{tabular}{|rp{0.35\textwidth}|rp{0.35\textwidth}|}\hline
	(a1)& $\dif \in \ker(\sigma)$;
	&
	(b1)& $\dif \in \ker(\rho_{v})$;
\\ \hline
	(a2)& $\dif$ fixes each vertex of $\XX$, and also leaves invariant every edge $\XX$ and preserves orientation.
	&
(b2)& $\dif$ leaves each 2-cell $\Sp$ invariant.
\\  \hline
\end{tabular}
\end{center}

Then it follows from the definitions of $\rho_{v}$ and $\sigma$ that (a1)$\Leftrightarrow$(a2) and (b1)$\Leftrightarrow$(b2).

Moreover, (a2) implies (b2), because each $2$-cell is uniquely determined by the edges to which it is glued and therefore $\ker(\sigma)\subset\ker(\rho_{v})$.

Conversely, (b2) implies (a2) due to~\cite[Theorem~7.1]{Maksymenko:AGAG:2006}.
Hence $\ker(\rho_{v})\subset\ker(\sigma)$ as well, and so $\ker(\rho_{v})=\ker(\sigma)$.
Thus $\lG$ is isomorphic to a certain subgroup of $\Homeo(\Sp)$.

\section{Proof of~\ref{enum:gloc:FixGf_edge} and~\ref{enum:gloc:FixGf_vertex} of Theorem~\ref{th:cw3}}
\label{sect:proof:th:cw3:bc}
Now let
\[
  \emb:
  \lG  \ \stackrel{\mu}{\longrightarrow} \
  \Is(\XX,j_{1})  \ \subset  \
  \Homeo^{\cw}(\XX,j_{1}) \ \stackrel{s}{\longrightarrow} \
  \Homeo^{\cw}(\Sp) \ \subset  \
  \Homeo(\Sp).
\]
be the embedding constructed in~\ref{enum:gloc:Gvloc_in_SO3} of Theorem~\ref{th:cw3}.
We will discuss its properties.
To simplify notations for each $\gel \in \lG$ denote its image in $\Homeo(\Sp)$ by $\sgel$, that is $\sgel := \emb(\gel)$.

\begin{lemma}\label{lm:h_and_gamma}
\begin{enumerate}[label={\rm(\arabic*)\,}, wide]
\item\label{enum:act:cellular_action}
Each $\sgel\in \emb(\lG) \subset \Homeo(\Sp)$ is a cellular homeomorphism of $\Partit$.
Also $\sgel = \id_{\Sp}$ if and only if $\sgel$ leaves invariant every $2$-cell of $\Partit$.

\item\label{enum:act:similar_to_stab}
If $\dif\in\GStabilizer{\func}$ and $\gel=\rho_{v}(\dif)\in\lG$, then $\dif(e) = \sgel(e)$ for each cell $e\in\Partit$.

\item\label{enum:act:z_e}
In each cell $e\in\Partit$ we can choose a point $z_e \in e$ such that their collection $\{ z_e \}_{e\in\Partit}$ is $\emb(\lG)$-invariant.
In particular, if $\sgel(e)=e$ for some $\gel\in\lG$, then $\sgel(z_e)=z_e$.

\item\label{enum:act:hprop}
Let $e$ be a cell in $\Partit$ and $A_e$ be the subgroup in $\emb(\lG)$, which leaves $e$ invariant.
\begin{enumerate}[label={\rm(\alph*)}, wide]
\item\label{enum:act:hprop:0}
Suppose that $\dim{e}=0$, so $z_{e}=e$ is the critical point $\func$ which lies on the critical level $\XX$.
Let also $\alpha_0,\ldots,\alpha_{2k-1}$ be cyclically ordered arcs on $\XX$ starting from point $z_e$ for some $k\geq1$ and having the same length $d < 0.5$, see Figure~\ref{fig:saddle_sectors}.
Then $A_e$ is a cyclic group freely acting on the set of arcs $\alpha_0,\alpha_2,\ldots,\alpha_{2k-2}$ with even indices.
In particular, its order divides $k$.

\item\label{enum:act:hprop:1}
If $\dim{e}=1$, that is $e$ is an edge $\XX$, then $A_e = \{ \id_{\Sp} \}$ is a unit group.

\item\label{enum:act:hprop:2}
Suppose $\dim{e}=2$.
Let also $\psi_{e}:D^2 \to \XX$ be the characteristic mapping $2$-cell $e$ and $F\subset S^1 = \partial D^2$ be that finite subset such that the corresponding gluing map $\psi_{e}|_{S^1}$ satisfies the conditions of Theorem~\ref{th:cw2}.
Denote by $n$ the number of points in $F$, which coincides with the number of arcs in $S^1\setminus F$ as well as with the number connected componets of $\overline{e}\setminus e$.
Then $A_e$ is a cyclic group freely acting on that set of arcs.
In particular, its order divides $n$.
\end{enumerate}

\item\label{enum:act:2cells}
If $\sgel\in \emb(\lG)$ is not the identity mapping, then $\sgel$ has exactly $2$ invariant cells $e_1, e_2 \in \Partit$, which may have different dimensions.
In particular, $\sgel$ has two fixed points $z_{e_1}$ and $z_{e_2}$.
\end{enumerate}
\end{lemma}
\begin{proof}
\ref{enum:act:cellular_action}
By the construction $\emb(\lG)$ consists of cellular homeomorphisms of $\Sp$.
Also notice that there exists a canonical bijection between $2$-cells of $\Partit$ and edges of the star $Star(v)$ of $v$ so that for each $\gel\in\lG$ the following conditions are equivalent:
\begin{enumerate}[label={(\roman*)}]
\item
$\sgel \in \Homeo(\Sp)$ leaves invariant every $2$-cell of $\Partit$;
\item
$\gel$ leaves invariant each edge of $Star(v)$;
\item
$\gel$ is the identity homeomorphism of $Star(v)$;
\item
$\sgel = \emb(\gel) = \emb(\id_{Star(v)}) = \id_{\Sp}$.
\end{enumerate}

\ref{enum:act:similar_to_stab}
This statement also follows from the construction of the embedding $\emb$.

\ref{enum:act:z_e}
In each $e$ we should choose a point $z_e \in e$ such that if $\sgel(e)=e$ for some $\sgel\in \emb(\lG)$, then $\sgel(z_e)=z_e$.

\begin{enumerate}[label=\alph*), wide]
\item\label{enum:act:z_e:0}
If $\dim e =0$, i.e. $e$ is a critical point of $\func$ belonging to $\XX$, then we must put $z_{e}=e$.

\item\label{enum:act:z_e:1}
Suppose $\dim e =1$, so $e$ is an edge $\XX$.
Recall that we have choosen a metric on $\XX$ in which each edge (in particular $e$) has length $1$.
Let $z_e$ be the middle point of $e$, so it splits $e$ into two arcs each of length $\tfrac{1}{2}$.
Hence if $\sgel(e)=e$, then $\sgel|_{e}$ is either the identity map, or is a unique reversing orientation isometry.
In both cases, $\sgel(z_e) = z_e$.

\item\label{enum:act:z_e:2}
Let $\dim e = 2$ and $\psi_{e}:D^{2}\to \XX$ be the characteristic mapping of the cell $e$ constructed in~\ref{enum:cw2:2dim} of Theorem~\ref{th:cw2}.
Put $z_e = \psi_e(0)$, where $0 \in D^2 \subset\bR^2$ is the origin.
Then it follows from formulas~\eqref{equ:beta} and~\eqref{equ:s} that $\sgel(z_e)=z_e$ whenever $\sgel(e)=e$.
\end{enumerate}

\ref{enum:act:hprop}
Let $e$ be a cell in $\Partit$.
We need to compute the subgroup $A_e \subset \emb(\lG)$ of homeomorphisms that leave $e$ invariant.

Let $\sgel\in A_e \in \emb(\lG)$, so there exists $\dif\in\GStabilizer{\func}$ such that $\gel=\rho_{v}(\dif)\in\lG$.
In particular, $e$ is a $\dif$-invariant cell of $\sgel$ and $z_{e}\in e$ is the corresponding fixed point of $\sgel$.

\medskip

\ref{enum:act:hprop:1}
Suppose $\dim e = 1$.
We will show that $\dif$ leaves each $2$-cell invariant.
Then due to~\ref{enum:act:cellular_action} we will get that $\sgel = \id_{\Sp}$, which will prove that $A_e = \{\id_{\Sp}\}$.

First we claim that {\em $\dif$ preserves orientation of $e$}.
Indeed, since $e$ is an edge, i.e. a part of the critical component of a some level set of $\func$, it follows that $e$ belongs to the closure of precisely two $2$-cells $\alpha, \beta \in \Partit$, and $\func(a) < \func(e) < \func(b)$ for all $a\in \alpha$ and $b\in\beta$.
But $\dif$ preserves $\func$, whence $\dif(\alpha)=\alpha$ and $\dif(\beta)=\beta$.
Moreover since $\dif\in\GStabilizer{\func} \subset \Homeo^{+}(\Sp)$ also preserves orientation of $\Sp$, it must preserve orientations of open subsets $\alpha$ and $\beta$ of $\Sp$, and therefore $\dif$ also preserves the orientation of $e$.

Hence $\dif$ fixes each vertex $x \in \overline{e}\setminus e$ of the edge $e$ being a saddle critical point of $\func$.
Let $k_x$ be the order of $x$, see Figure~\ref{fig:saddle_sectors}, so there are $2k_x$ arcs in $\XX$ starting from $x$ which are cyclically ordered and $\dif$ preserves their cyclic order.
Therefore, $\dif$ leaves invariant all the edges of $\XX$ incident to $x$.

This implies that the closure $B$ of the set of $\dif$-invariant edges is open in $\XX$.
Then it follows from the connectedness of $\XX$ that $B = \XX$, i.e.\! all edges of $\XX$ are invariant with respect to $\dif$.
Therefore, $\dif$ leaves invariant each $2$-cells, whence $\gel = \id_{\Sp}$.

\medskip

\ref{enum:act:hprop:0}
Suppose $\dim{e}=0$.
Since $\dif$ preserves the cyclic order of the arcs $\alpha_0,\ldots,\alpha_{2k-1}$, it follows that $\dif(\alpha_{i})=\alpha_{i+\eta}$ for some $\eta\in\{0,\ldots,2k-1\}$.
Not loosing generality one can assume that $\func$ equals to $0$ on $\XX$.
Then on the consecutive sectors $\widehat{\alpha_{i-1} \alpha_{i}}$ and $\widehat{\alpha_{i} \alpha_{i+1}}$ the function $\func$ takes values of different signs.
As $\dif$ preserves the values of $\func$, it follows that $\eta$ must be even, so $\dif(\alpha_{i})=\alpha_{i+ 2\tau}$ for some $\tau\in\{0,\ldots,k-1\}$.
Therefore, the set of $k$ arcs with even numbers are also invariant with respect to $\gel$.

In other words, we get an action of the group $A_e$ on the set of arcs with even numbers by cyclic shifts, which can be viewed as a homomorphism $q:A_e \to \bZ_{k}$.
Note that when $\gel(\alpha_i) = \alpha_i$ for some $i\in\{0,2,\ldots,2k-1\}$, then according to the previous statement~\ref{enum:act:hprop:1} $\gel = \id_{\Sp}$.
This means that the action of the group $A_e$ is free and so $q$ is a monomorphism.
In other words, $A_e$ is a subgroup of $\bZ_{k}$, whence it is cyclic and its order divides $k$.

\medskip

\ref{enum:act:hprop:2}
Suppose $\dim{e} = 2$ and let $n$ the number of points in $F$.
Then according to the diagram~\ref{equ:kd1} $\gel$ cyclically shifts $n$ edges $\delta_0,\ldots,\delta_{n-1}$ of $\XX$ along which $e$ is glued to $\XX$.
Therefore, similarly to the previous paragraph~\ref{enum:act:hprop:0} we get a free action of $A_e$ on the set of those edges by cyclic shifts, which implies that $A_e$ is also cyclic and its order divides $n$.

\ref{enum:act:2cells}
According to \cite[Corollary 5.4]{Maksymenko:MFAT:2010} a cellular homeomorphism $\sgel$ of a closed orientable surface $\Mman$ with a cell partition $\Partit$
\begin{itemize}
\item either leaves every cell invariant and preserves its orientation;
\item or the number of invariant cells is equal to the Lefschetz number $L(\sgel)$ of $\sgel$.
\end{itemize}
In particular, in our case, $\gel$ is a preserving orientation homeomorphism of $S^2$.
Therefore, $\sgel$ is isotopic to the identity and so $L(\sgel) = \chi(\Sp) = 2$.
Thus if $\sgel$ is not the identity, then it has precisely two invariant cells.
\end{proof}

Now we can prove~\ref{enum:gloc:FixGf_edge} and~\ref{enum:gloc:FixGf_vertex} of Theorem~\ref{th:cw3}.

\subsection*{\ref{enum:gloc:FixGf_edge}}
Suppose $\Fix{\fG}$ has a fixed edge.
Let $v$ be a vertex in $\Fix{\fG}$ and $\Partit$ be the cellular partition of $\Sp$ constructed in~\ref{enum:gloc:Gvloc_in_SO3} of Theorem~\ref{th:cw3}.
As $\Fix{\fG}$ is a tree, $v$ belongs to some edge $\delta \subset \Fix{\fG}$ corresponding to some $2$-cell $e$ of $\Partit$.
This cell is therefore invariant with respect to the action of $\emb(\lG)$ on $\Sp$, that is $\lG = A_e$.
But then, by \ref{enum:act:hprop}\ref{enum:act:hprop:2} of Lemma~\ref{lm:h_and_gamma} $\lG = A_e$ is a cyclic group.

\subsection*{\ref{enum:gloc:FixGf_vertex}}
Suppose $\Fix{\fG}$ consists of a unique vertex $v$ and $\lG \cong \bZ_k$ for some $k\geq2$.
As $\emb(\lG)$ is conjugated to a finite (cyclic) subgroup of $SO(3)$, all elements of $\emb(\lG)$ have exactly two common fixed points, which we will denote by $a$ and $b$.

On the other hand, according to~\ref{enum:act:2cells} of Lemma~\ref{lm:h_and_gamma} every nontrivial element $\sgel\in\emb(\lG)$ has exactly two invariant cells $e_1$ and $e_2$ respectively and therefore two fixed points $z_{e_1}$ and $z_{e_2}$.
Hence, one can assume that $a=z_{e_1}$ and $b=z_{e_2}$.
In particular, $\lG = A_{e_1} = A_{e_2}$.
Therefore, it suffices to consider the following three cases.

\begin{enumerate}[wide, label={\alph*)}, itemsep=1ex]
\item
If $\dim e_1=2$, then $\Fix{\fG}$ must have a fixed edge that corresponds to the $2$-cell $e$, which contradicts to the assumption that $\Fix{\fG}=\{v\}$ consists of a unique vertex.

\item
If $\dim e_1=1$, then according to \ref{enum:act:hprop}\ref{enum:act:hprop:1} of Lemma~\ref{lm:h_and_gamma} $\lG = A_{e_1}$ is a trivial group, so $\Fix{\fG} = \KRGraphf \not=\{v\}$, which again contradicts to the assumption.

\item
Thus the remained situation is when both $e_1$ and $e_2$ are vertices of $\XX$, being therefore saddle critical points of $\func$.
Let $k_i$, $i=1,2$, be the order of $e_i$.
Then $\lG = A_{e_i}$ is isomorphic to a subgroup of $\bZ_{k_i}$ for both $i=1,2$.
Hence $k$ divides both $k_1$ and $k_2$, and therefore $GCD(k_1,k_2)$.
\end{enumerate}


\def\cprime{$'$}

\end{document}